\documentclass[9pt]{article}
\usepackage{amsmath}
\usepackage{amssymb}
\usepackage{amsthm}

\usepackage{graphicx}

\graphicspath{ {./latek_pictures/} }
\title{Rainbow Combinatorial Lines in Hypercubes}
\author{Michael Zheng\footnote{This research was supported by NSF DMS grant no. 234968.}}
\newtheorem{theorem}{Theorem}[section]
\newtheorem{corollary}[theorem]{Corollary}
\newtheorem{lemma}[theorem]{Lemma}
\newtheorem*{corollary*}{Corollary}
\newtheorem*{observation}{Observation}

\newtheorem{definition}{Definition}[section]

\begin{document}
\maketitle

\pagenumbering{arabic}
\begin{abstract}
This paper is about the rainbow dual of the Hales Jewett number, providing general bounds an anti-Hales Jewett Number for hypercubes of length k and dimension n denoted $ah(k, n).$ The best general bounds this paper provides are: $(k-1)^n < ah(k, n) \leq \frac{(k-1)^2-2}{k-1}\cdot k^{n-1}+\frac{k+1}{k-1}.$ This paper also includes proofs about the specific cases of $k = 2$ and $k = 3$, where we show that $ah(2, n) = 2$ and $2^n < ah(3, n) \leq 3^{n-1} - 2\cdot3^{n-4} + 2$ for all natural numbers n $>$ 4. For $n < 4$, we have found the exact values: $ah(3, 1) = 3$, $ah(3, 2) = 5$, and $ah(3, 3) = 11$. In the case $n = 4$, we have found that $23 < ah(3, 4) \leq 27$.
\end{abstract}
Keywords and phrases: Hales-Jewett number, rainbow colorings, hypercubes, combinatorial lines.
%\tableofcontents%

%\listoffigures%

\section{Introduction} 
In 1963, A. W. Hales and R. I. Jewett published the eponymous "Hales-Jewett Theorem" which has since become a hallmark of Ramsey Theory \cite{naslund_hales-jewett_2013}. As the study of Ramsey Theory progressed, mathematicians like P. Erd\"{o}s, M. Simonovits, and V. T. S\'{o}s questioned whether there were rainbow duals of Ramsey Theory, which began the branch of Anti-Ramsey Theory%cite erdos
. Since their paper in 1975, mathematicians have made many discoveries in Anti-Ramsey Theory; for example, Young et al found a formula for anti-van der Waerden numbers (the rainbow analogue of van der Waerden numbers) of three term arithmetic progressions in \cite{berikkyzy_anti-van_2017}. Though classical theorems like the van der Waerden Theorem have been studied through a rainbow lens, not much work has been published about a potential anti-Hales Jewett Theorem. This paper will introduce the anti-Hales Jewett Number as a potential area of research in anti-Ramsey Theory. Analogous to the anti-Ramsey number, the anti-Hales Jewett Number will be the minimum number of colors in colorings of the vertices of a hypercube required to guarantee a rainbow combinatorial line. We will begin with some definitions and state the original Hales-Jewett Theorem, with k, n \(\in \mathbb{N}\). Note that an r-coloring of \([k]^n\) is a coloring of all the vertices in \([k]^n\) with r colors.
\begin{definition}[Hypercube Notation]
We will use $[k]$ to be the set $\{i \in \mathbb{N}: 1 \leq i \leq k\}$ and $[k]^n$ to be the hypercube of length k and dimension n. In other words, $[k]^n$ is a graph with $k^n$ vertices v, labeled as an n-dimensional vector $v = (v_1,v_2,...,v_n)$ for $v_i \in [k]$.
\end{definition}
\begin{definition}[Combinatorial Line]
Let w be a word of length n with letters \(\in [k]\cup\{*\}\), where at least one of the letters is a \(*\). Then, a combinatorial line formed by w contains the words \(w_1, w_2, ..., w_k\) for i \(\in [k]\) where each \(w_i\) is equivalent to w, except each \(*\) is replaced by i. We will also use these words to denote vertices of the hypercube. We will call a combinatorial line "rainbow" with respect to a coloring of $[k]^n$ if no color appears more than once on $w_1,..., w_k$. Whenever we refer to a "line" in this paper, assume we refer to a  combinatorial line unless stated otherwise.
\end{definition}
\begin{definition}[Rainbow Free]
We will call a r-coloring of \([k]^n\) "rainbow free" if there are no rainbow combinatorial lines.
\end{definition}
\begin{definition}[anti-Hales Jewett Number]
The anti-Hales Jewett Number for the hypercube \([k]^n\), denoted ah(k, n), is the minimum number of colors to guarantee that any ah(k, n)-coloring of \([k]^n\) contains a rainbow combinatorial line.
\end{definition}
\begin{definition}[Minimal Coloring]
A minimal coloring of $[k]^n$ occurs when all colors but one appear only once. We will call the color that appears multiple times the "dominant" color.
\end{definition}
\begin{theorem}[Hales-Jewett (1963)] \footnote{This text is paraphrased from \cite{naslund_hales-jewett_2013}}
For all values k, r $\in \mathbb{N}$, there exists a number HJ(k, n) such that if N $\geq$ HJ(k, n) and the points of $[k]^N$ are colored with r colors, then $[k]^N$ contains atleast one monochromatic combinatorial line.
\end{theorem}
In traditional Ramsey Theory, the Hales-Jewett and van der Waerden numbers are intricately connected; the Hales Jewett theorem is essentially an abstraction of van der Waerden's theorem \cite{naslund_hales-jewett_2013}\cite{lee_ramsey_2009}. Paralleling their monochromatic counterparts, the anti-Hales Jewett numbers seem to be linked to anti-van der Waerden numbers. The bounds we find in this paper look very similar to the anti-van der Waerden numbers of arithmetic progressions, specifically the connections between 3 term arithmetic progressions and hypercubes with side length 3. I believe these two numbers naturally connect since there is always an isomorphism from a hypercube $[3]^n$ to $[3^n]$ that lets combinatorial lines be written as 3 term arithmetic progressions. However, the Anti-van der Waerden number will always be larger since combinatorial lines are more restrictive than arithmetic progressions. For example, we can consider the map $f: [3]^2 \to [9]$ to be defined by $f(xy) = 3\cdot(x-1) + y$, where xy is a word in $[3]^2$. We then notice that the arithmetic progression ${3, 5, 7}$ in $[9]$ corresponds to the words $13, 22,$ and $31$, which is not a combinatorial line. I conjecture that for geometric lines in hypercubes, an analogue of the anti-Hales Jewett number would be equal to the anti-van der Waerden numbers, which is another reason why I only consider combinatorial lines in this paper. 
%maybe add what the bounds are for aw numbers
\section{Bounds for the anti-Hales Jewett Number}
There are two types of bounds for the Anti-Hales Jewett Number, the recursive bounds and the non-recursive bounds. As the numbers for higher and higher dimensions are found, the recursive upper bound becomes better and better. However, since we don't know many exact values of ah(k,n), the non-recursive bounds provide an informative estimate. 
\begin{theorem} 
Let k, n \(\in\) \(\mathbb{N}\), with k $> 2$ and n $> 1$. \(Then, (k-2)\cdot ah(k, n - 1) - k + 3 < ah(k, n) \leq k\cdot ah(k , n-1) - k - 1.\) 
\end{theorem}
\begin{theorem}
Let k, n \(\in\) \(\mathbb{N}\), with k, n $> 2$. \(Then, (k-1)^n < ah(k, n) \leq \frac{(k-1)^2-2}{k-1}\cdot k^{n-1}+\frac{k+1}{k-1}.\) 
\end{theorem}
Notably, the nonrecursive bound for k = 3, the main case studied in this paper, does not involve any fractions since $k - 1 = 2$ and $k - 1 | k + 1$.
\begin{corollary}
Let n \(\in \mathbb{N}\). Then, \(2^n < ah(3, n) \leq 3^{n-1} + 2.\)
\end{corollary}
\subsection{Proofs}
We will first note two observations that will assist in the proofs of the general bound. I will use the term "disjoint" to reference subgraphs which have no overlap in \([k]^n\), meaning no vertices are shared.
\begin{observation}
Let \(k \in \mathbb{N}\). Then the hypercube \([k]^n\) contains k disjoint copies of \([k]^{n-1}\).
\end{observation}
You can imagine these copies of \([k]^{n-1}\) as layers "stacked" on top of each other. For $i \in [k]$ and some t $\in [k]$, we will define $L_i = L(i, t, k ,n) = \{w_1...w_{t-1}iw_{t+1}...w_k : w \in [k]^n\} \in [k] \cong [k]^{n-1}$. The second observation is as follows:
\begin{observation}
$ah(k, 1) = k.$
\end{observation}
\begin{lemma}
If we split \([k]^n\) into $L_1, L_2, .... L_k$ and a combinatorial line that contains points in $L_i$ and $L_j$ for $i \neq j$, then the combinatorial line contains points in $L_h$ for all $h\in[k]$.
\end{lemma}
\begin{proof}
Suppose that a combinatorial line contains points in $L_i$ and $L_j$ for $i \neq j$, say points v and u. Then $v_t \neq u_t$; more specifically, $v_t = i$ and $u_t = j$. This tells us that there is a $*$ at position t and since the combinatorial line is formed by replacing $*$ with every element of $[k]$, there exists a point w such that $w_t = h$ for all $h \in [k]$. 
\end{proof}
With the observations and lemma, we can easily prove the recursive bounds in Theorem 2.1.
\begin{proof}
For the lower bound, we first note that \(\exists\) a rainbow free \((ah(k, n-1)-1)\)-coloring of \([k-1]^n\) by definition of \(ah(k, n-1)\). We will construct a rainbow-free coloring of \([k]^n\) that involves \((k-2)\cdot ah(k, n-1) - k + 3\) colors. First, we color layers $L_1, L_2, ... L_{k-2}$ with \(ah(k, n-1)-1\) distinct colors such that there are no rainbow combinatorial lines on each copy. We will then choose a new color c to create a monochromatic coloring of $L_{k-1}$ and $L_k$. We now have a \((k-2)\cdot(ah(k, n-1)-1) +1 = (k-2)\cdot ah(k, n-1) - k + 3\) coloring of \([k]^n\). This coloring is rainbow free; each individual $L_i$ contains no rainbow combinatorial lines. The remaining lines we need to check are those that contain exactly one point in each layer by Lemma 2.4. These lines will always have two points colored c, and thus are not rainbow. \\
\indent \indent 
For the upper bound, we must show that every \(k \cdot ah(k , n-1) - k - 1\) coloring of \([k]^n\) will contain a rainbow combinatorial line. Seeking a contradiction, we assume there exists a $k(ah(k, n-1) -1) - 1$ coloring of $[k]^n$ that is rainbow free. Let $C_i$ be the color set of $L_i$, i.e., $|C_i|$ is the number of different colors appearing on $L_i$. Then, we know that $|C_i| \leq ah(k, n-1) - 1 \ \forall i \in [k]$ since the coloring is rainbow free. Thus the total number of colors $|C| = |C_1| \cup |C_2| \cup ... \cup |C_k| \leq k(ah(k, n-1) -1) - 1$. There are two possible cases for when $|C| = k(ah(k, n-1) -1) - 1$:\\
1. The $C_i$ are pairwise disjoint, which means $|C_j| = ah(k, n-1) - 2$ for some $j \in [k]$ and the rest of the $|C_i| = ah(k, n-1) - 1$.\\
2. All $|C_i| = ah(k, n-1) -1$ and are pairwise disjoint except for exactly one element in C, which is shared by exactly two of the $C_i$'s. We will denote this element as c'.\\
In the first case, we immediately see that if all $C_i$ are pairwise disjoint, then any combinatorial line that contains a point on each $L_i$ will be rainbow, which contradicts our assumption. Thus, we only need to consider the second case. \\
Let us consider an arbitrary combinatorial line with points $v_1, v_2, ..., v_k$ where each $v_i \in L_i$. Since our coloring is rainbow free, there must be some $v_i$ and $v_j$ colored with the same color. This color must be c', as it is the only color shared between layers. Since all possible points $v \in L_i \cup L_j$ can be part of a combinatorial line that contains a point from every layer, each $v$ must be colored with c'. This implies $|C_i| = |C_j| = 1$. But, $|C_i| = ah(k, n-1) - 1$, implying $ah(k, n-1) = 2 \geq ah(k, 1) = k$, contradicting $k, n > 2$. Thus, there does not exist a rainbow free coloring of $[k]^n$ with $k \cdot ah(k, n-1) - k - 1$ colors.

\end{proof}

Now, we will prove Theorem 2.2.
\begin{proof}
For the lower bound, we will construct a rainbow free coloring of $[k]^n$ with $(k-1)^n$ colors. Our coloring scheme will involve the specific digits of the words/points in the hypercube, and we will color two points the same color if and only if the digits for $1, 2, ... , k - 2$ are in the same spot. There can be multiple occurances of each number, and so long as the indices of those numbers are the same, then the two points will be the same color. For example, let $s_1, s_2 \in [4]^3$ be the points $1123$ and $1124$. Since $s_1$ and $s_2$ both have "1" for the first two digits and "2" for the third, they will share the same color. Notice that this coloring scheme will always prevent rainbow combinatorial lines, as atleast two points along the same line will have the same color. Consider an arbitrary combinatorial line $\ell$ based on the word $x_1x_2...x_n$, with some $x_{i1}, x_{i2}, ... x_{ij} = *$. Then, the two points in $\ell$ when $* = k - 1$ and $* = k$ will have the same color. Indeed, the indices for $1, 2, ..., k - 2$ will be the exact same, since the only difference between the words is the $*$ changing from $k - 1$ to $k$. Thus, we have constructed a rainbow free coloring of $[k]^n$ with $(k-1)^n$ colors. \\
\indent \indent The upper bound will use our observation that $ah(k, 1) = k$ and continuously apply the recursive upper bound on the known value of ah(k, 1). Then, the upper bound for $[k]^n$ will take the form of: \[k\cdot(k\cdot(k\cdot....(k\cdot(k) - k - 1) ... - k - 1) - k - 1) - k - 1. \] Expanding this out and compressing using summation notation, we get this expression: \[k^n - k^{n-1} - 2\cdot\sum_{i = 1}^{n-2} k^i - 1.\]  Applying the geometric summation formula and simplifying gives us: 
\begin{equation*}
\begin{split}
k^n - k^{n-1} - 2\cdot(\frac{k^{n-1} - 1}{k - 1} + 1) -1 &= (k-1)k^{n-1} - 2\cdot\frac{k^{n-1} - 1}{k - 1} + 1 \\ 
&= (k-1)k^{n-1} - 2\cdot\frac{k^{n-1}}{k-1} + \frac{2}{k-1} + 1 \\
&= k^{n-1}\cdot((k-1) - \frac{2}{k-1}) + \frac{k+1}{k-1} \\
&= \frac{(k-1)^2-2}{k-1}\cdot k^{n-1}+\frac{k+1}{k-1}.
\end{split}
\end{equation*}
We have now proved Theorem 2.2, and by implication, Corollary 2.3.
\end{proof}
\begin{corollary}
Suppose we know $ah(k, m) \leq a_m$. Then for $n > m$, $ah(k, n) \leq \frac{(a_m - 1)(k - 1) - 2}{k - 1}\cdot k^{n-m} + \frac{k + 1}{k - 1}$.
\end{corollary}
This can be proven using the same expansion and simplification as the general upper bound, but this corollary importantly tells us that if we find a large $a_m$ that is less than the general upper bound, we will get a much better bound for future anti-Hales Jewett Numbers.
\section{Anti-Hales Jewett Number for 2-Hypercubes}
The anti-Hales Jewett number for 2-hypercubes is very easy to show through induction, and so we arrive at the following theorem. Note that we will abbreviate "rainbow free" as "RF."
\begin{theorem}
For all n $\in \mathbb{N}, ah(2, n) = 2$.
\end{theorem}
\begin{proof}
We will begin with the base case: n = 1. If n = 1, then we already know that $ah(k, 1) = k$ implies $ah(2, 1) = 2$. \\
Now assume the theorem is true for all $n' < n$. Then, we know that $ah(2, n-1) = 2$.  Suppose we have a RF-coloring of $[2]^n$ with r colors (where $r = ah(2,n) - 1$). Then since $[2]^n$ contains two layers of $[2]^{n-1}$ ($L_1$ and $L_2$), we have that $r \leq 2$, since each $L_i$ must contain one color. Since the length of combinatorial lines in $[2]^n$ is 2 for all $n \in \mathbb{N}$ and that every line is rainbow free, lines that contain points on $L_1$ and $L_2$ must be monochromatic. Therefore, $L_1$ and $L_2$ must be colored monochromatically with the same color, which means $r = 1$.\\
By induction, we know this to be true for all natural numbers n.
\end{proof}
\section{Anti-Hales Jewett Number for 3-Hypercubes}
Recall Corollary 2.3:
\begin{corollary*}
Let n \(\in \mathbb{N}\). Then, \(2^n < ah(3, n) \leq 3^{n-1} + 2.\)
\end{corollary*}
Interestingly, the upper bound for $k = 3$ and $n = 2, 3$ is sharp: $ah(3, 2) = 5$ and $ah(3, 3) = 11$. However, for $n = 4$, the upper bound does not seem to be sharp; I have only found up to a 23 coloring of $[3]^4$ that contains no rainbows. To prove $ah(3, 2) = 5$ and $ah(3, 3) = 11$, it will suffice to prove that $\exists$ a rainbow free 4 coloring of $[3]^2$ and a rainbow free 10-coloring of $[3]^3$.\\
\begin{figure}[h]
    \centering
    \includegraphics[width=0.4\textwidth]{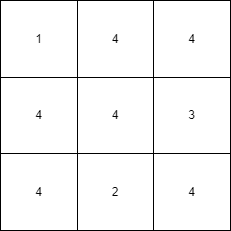}
    \caption{A Rainbow-Free 4-Coloring of $[3]^2$}
    \label{fig:coloring1}
\end{figure}
In Figure \ref{fig:coloring1}, the word $yz$ corresponds to the tile on the y'th row and z'th column. In this case, the labeling would not necessarily matter, but for the three dimensional case, it may. \\
\begin{figure}[h]
    \centering
    \includegraphics[width=1.0\textwidth]{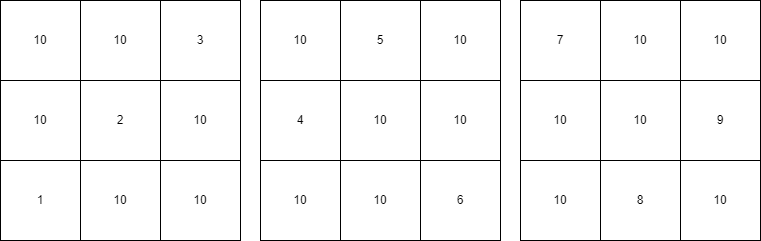}
    \caption{A Rainbow-Free 10-Coloring of $[3]^3$}
    \label{fig:coloring2}
\end{figure}
We maintain the same labeling of the word in Figure \ref{fig:coloring2}, except adding a first digit/letter to reference the layer $yz$ is on; the word $xyz$ refers to the tile on the x'th layer, y'th row, and z'th column, with layers ascending from left to right. \\
Although we do not currently know the anti-Hales Jewett Number for dimensions above 3, we do have a better bound than the one given by Theorem 2.2, by using Corollary 2.5 and some new Lemmas to show that there are only two rainbow free 10-colorings of $[3]^3$, which we will use to show the upper bound $ah(3, 4) \leq 27$. Furthermore, we can sharpen the upper bound for all $n > 4$ as well with Corollary 2.3. For the following proofs, we will consider an arbitrary partition of $[3]^3$ into three disjoint copies of $[3]^2$ and refer to them as layers $L_1, L_2, L_3$. Unlike before, these layers are arbitarily ordered unless specified to maintain generality. 
\begin{lemma}
Consider an arbitrary RF 10-coloring of $[3]^3$. Then for any partition of $[3]^3$ into three layers of $[3]^2$, each layer contains 3 distinct colors and all share exactly 1 color with each other. 
\end{lemma}
\begin{proof}
We will first show that every copy of $[3]^2$ contains at most three distinct colors in a RF 10-coloring of $[3]^3$. For the sake of contradiction, suppose $\exists$ a partition in which a layer of $[3]^2$ contains 4 distinct colors. which we will designate $L_1$. Then the remaining 6 colors must be split among the other two layers, $L_2$ and $L_3$, since the coloring would no longer be RF if $L_1$ contains more than 4 colors. Notice that this implies some layer contains at least one distinct color (in actuality, it must be two distinct colors, but that observation is unnecessary for this proof) since we cannot have a layer with 5 colors. Let this layer be $L_2$. But if we consider the combinatorial line that contains points in all three layers which includes the distinctly colored point in $L_2$, we see that it must be rainbow, since two of the points are distinctly colored, meaning that the colors do not appear outside of their respective layers. Thus we contradict the RF assumption, so every layer must contain at most 3 distinct colors. \\
\indent \indent Now, we will show that there cannot be two copies which share more than one color with each other, which will help us prove that no copy can have less than 3 distinct colors. Again, for the sake of contradiction, assume $\exists$ two layers, $L_1$ and $L_2$, that share more than one color with each other. Sharing 3 colors with each other would result in a layer containing more than 4 colors, since 7 colors must be split among the 3 layers, among which $L_1$ and $L_2$ already contain 3 colors (meaning they can only hold one more color). Sharing 4 colors is essentially the same; a layer would have more than 5 colors, which contradicts the RF assumption. Thus, if two layers share more than one color with one another, they must share exactly two colors with each other. Suppose $L_3$ has $0 \leq j \leq 3$ distinct colors. Then, any new color we add to the third layer must be shared with at least one of $L_1$ or $L_2$. Then, the remaining number of colors, $10 - j - 2$, must be split among the first two layers. But since $5 \leq 10 - j - 2 \leq 8$ and the first two layers are already sharing two colors, one of these two layers must always contain more than 4 colors, and we have a contradiction to the RF assumption. Thus, two layers cannot share more than one color with one another. \\
\indent \indent Now, we have the properties to prove that a layer cannot contain less than 3 distinct colors. The cases where a layer contains 0 distinct colors or 1 distinct color are trivial; in the former, that would imply 10 colors are split among the other two layers, guaranteeing a layer containing 5 or more colors. Similarly, in the latter, 9 colors split among the other two layers would also guarantee a layer containing 5 or more colors by the pigeonhole principle. The only case we would need to deal with is if a layer contains 2 distinct colors. For contradiction, suppose that we have a RF 10-coloring of $[3]^3$ with a layer, $L_3$, containing 2 distinct colors. Now, let $j_1,j_2 \in \{0, 1\}$ be the number of colors shared with $L_1$ and $L_2$, respectively. Then the number of distinct colors split between $L_1$ and $L_2$ is: $10 - j_1 - j_2 - 2$, which is 6, 7, or 8.  Note that in the case where the color shared with the other layers is the same, we subtract one, which will still result in a value of 6, 7, or 8. We notice in the case of 7 and 8 that one layer must contain four distinct colors by the pigeonhole principle, which contradicts the RF assumption. We now must deal with the case where 6 distinct colors are split between $L_1$ and $L_2$, i.e., when $L_3$ shares one color with $L_1$ and another with $L_2$. This means that $L_3$ has 4 colors total. With 6 colors that need to collectively appear on $L_1$ and $L_2$, we also know that $L_1$ and $L_2$ cannot share a color; otherwise, that color must not appear on $L_3$ and $L_1$ or $L_2$ is forced to contain more than 4 colors, which contradicts the RF assumption. Then, $L_1$ and $L_2$ must each contain colors distinct from one another, which means that $L_1$ and $L_2$ also contain 4 colors. But this coloring must contain a rainbow since $L_3$ contains distinct colors; any combinatorial line containing a distinctly colored point on $L_3$ and has two other points on $L_1$ and $L_2$ will be rainbow. Thus, we have a contradiction and each layer must contain exactly 3 distinct colors. \\
\indent \indent The three layers all sharing one color also follows from the previous statements; if any two layers share no colors, then any combinatorial line spanning the 3 layers containing a distinctly colored point on the third layer would be rainbow. We have now proven the lemma.
\end{proof}
With this crucial property of RF 10-colorings of $[3]^3$ established, we can now make these "arbitrary" 10-colorings more specific to narrow down the number of possible colorings.
\begin{lemma}
A RF 10-coloring of $[3]^3$ is always minimal with the shared color being dominant. 
\end{lemma}
\begin{proof}
%By Lemma 3.1, we know that each layer has 3 distinct colors, so we will look at the combinatorial lines that contain a distinct color and span the "pillars" of each layer. We will use the ordered $L_i$'s in this proof; i.e. $L_1$ consists of all points with first letter "1." Let $ix_2x_3$ designate a distinctly colored point. Then the combinatorial line defined by $*x_2x_3$ must not be rainbow, and thus, the remaining two points on the other non $L_i$ layers must be the same color. Since there is only one shared color between layers, the remaining two points must contain that color. This is true for every combinatorial line of this form, which we will call "pillar lines." Since there are 9 pillar lines (the number of possible $*x_2x_3$ is $3\cdot3 = 9$) and 9 distinct colors, each color must be assigned to a unique pillar line and occur exactly once on each line. Thus, every color but the shared color occurs only once, and thus, the RF 10-coloring is minimal.%
By Lemma 4.1, we know that each layer contains 3 distinct colors and share one color, which we will denote c, with each other. Let $[3]^3 = L_1 \cup L_2 \cup L_3$ where $L_i = {ix_2x_3 : x_2, x_3 \in [3]}$. We will define the combinatorial line described by $*x_2x_3$ to be a "pillar line." Since every pillar line must be RF, two points $jx_2x_3$ and $kx_2x_3$ on the pillar line must be the same color, implying the color must be c.  Thus for every distinctly colored point, the associated pillar line must contain two points colored c. In other words, there are 18 points colored with c and thus. Since there are 9 distinct colors and 9 points remaining, there must be exactly 9 distinctly colored points. Thus, the shared color is dominant and the 10-coloring is minimal. 
\end{proof}
Lemma 4.2 essentially tells us that each layer is also a minimal coloring of $[3]^2$ with 4 colors, which we can easily group into two possible sets.
\begin{lemma}
There are two sets of minimal 4-colorings of $[3]^2$, $S_1$ and $S_2$, shown below.
\end{lemma}
\begin{figure}[h]
    \centering
    \includegraphics[width=0.5\textwidth]{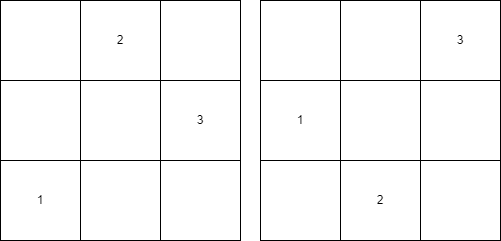}
    \caption{$S_1$}
    \label{fig:coloring3}
\end{figure}
\begin{figure}[h]
    \centering
    \includegraphics[width=0.75\textwidth]{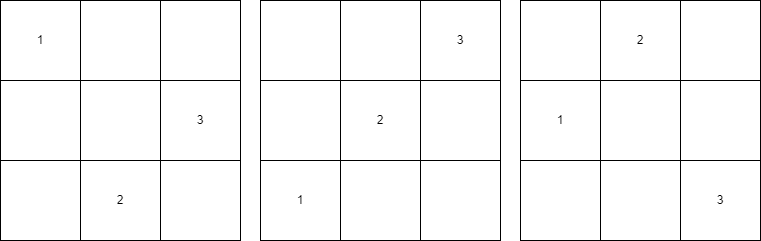}
    \caption{$S_2$}
    \label{fig:coloring4}
\end{figure}
\begin{proof}
Since there is a dominant color, we can algorithmically show that there are very few minimal RF 4-colorings of $[3]^2$. We will choose a starting point which will be colored with a nondominant color, then eliminate points along the same combinatorial line as the start, then repeating with the remaining points until we have a complete coloring. Since we are in two dimensions, there are two types of starting points: a point part of 3 combinatorial lines or a point part of 2 combinatorial lines. \\
\indent \indent In the first case, we immediately receive a minimal RF 4-coloring, since 7 points will have been colored after the first step. Since the start point is part of 3 combinatorial lines, one of the combinatorial lines is the diagonal. Thus, out of each column and row, atleast two points will have been colored. This means the remaining 2 uncolored points are not along the same row or column, and neither are part of the diagonal. Thus, coloring these two points with distinct colors results in a RF minimal coloring. These colorings are shown in Figure \ref{fig:coloring5}.\\
\begin{figure}[h]
    \centering
    \includegraphics[width=1.0\textwidth]{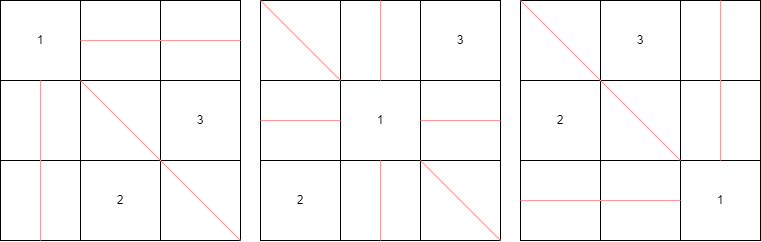}
    \caption{Diagonal Start Points}
    \label{fig:coloring5}
\end{figure}
\indent \indent In the second case, we will apply the same algorithm. There are 4 uncolored points after the first step, which can be represented by a 2 by 2 grid based on the rows and columns they share. This square diagram will always be accurate, since if the point is part of two combinatorial lines, then those lines will be the row and column. Thus, out of each row and column that the starting point is not part of, one point will already be colored and two will still be uncolored. Since we cannot color two points along the same row or column, we have only two possibilities, which are the two diagonals of the grid. We now have the complete list of colorings, shown in Figure \ref{fig:coloring6} and contains many repeats. All of which are either a part of $S_1$ or $S_2$.
\end{proof}
\begin{figure}[h]
    \centering
    \includegraphics[width=0.75\textwidth]{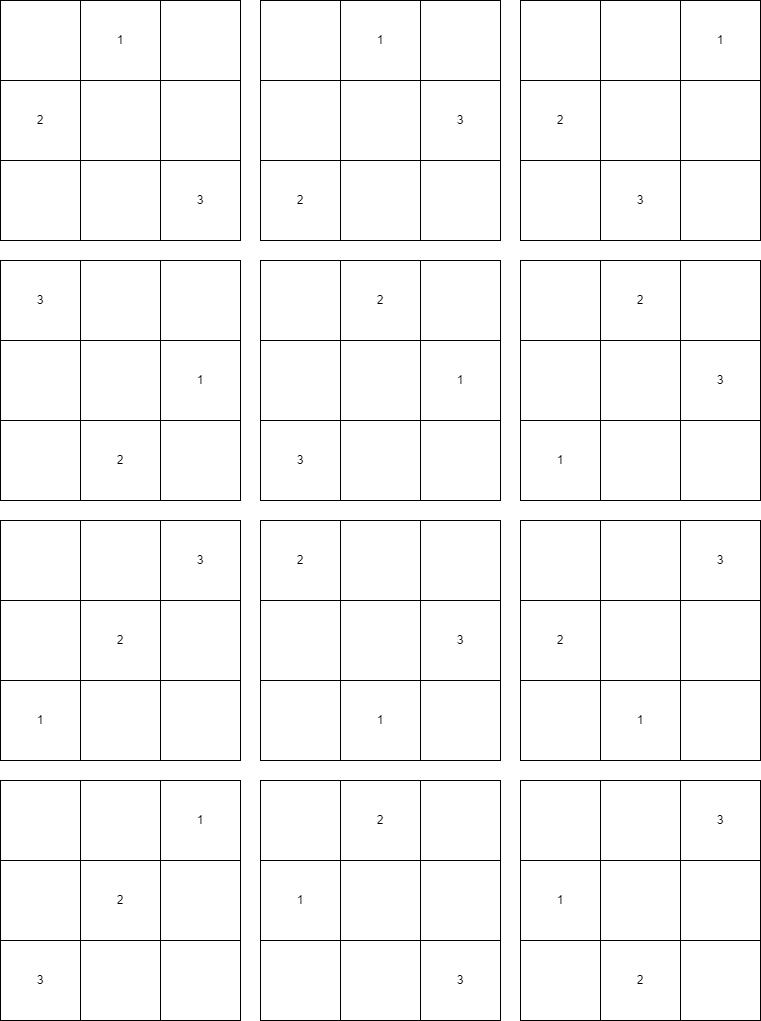}
    \caption{Non-Diagonal Start Points}
    \label{fig:coloring6}
\end{figure}
\begin{lemma}
Consider any RF 10-coloring of $[3]^3$. Then for any partition of $[3]^3$ into 3 layers of $[3]^2$, each layer is a minimal 4-coloring from $S_2$.
\end{lemma}
\begin{proof}
We know that every layer is either from $S_1$ or $S_2$. Notice that we cannot mix colorings between sets; if a layer is in $S_1$, then another layer being in $S_2$ will always result in rainbow pillar line. But if we have every layer being a coloring from $S_1$, which contains two elements, then atleast two of the layers have the same coloring. This coloring will always contain a rainbow pillar line since two distinct nondominant colors will appear on the same combinatorial line, contradicting the RF assumption. 
\end{proof}
\begin{lemma}
There are two RF 10-colorings of $[3]^3$.
\end{lemma}
\begin{proof}
Note that the total number of combinations of layers from $S_2$ in $[3]^3$ would be $3! = 6$. Stacking these in any arrangement would never result in rainbow pillar lines, so we must look at diagonals. We will order the layers the same way as in Lemma 3.2, i.e. $L_i$ refers to all points of the form $ix_2x_3$. Suppose $L_1$ is $s_1$. Consider the combinatorial line defined by $**1$. This line must be not be rainbow, thus the second layer cannot be $s_2$. Then the second layer must be $s_3$. But then the line $***$ would be rainbow. Thus, $L_1$ cannot be $s_1$. Now let $L_1$ be $s_2$. Then, $L_2$ cannot be $s_1$, or else the lines $*3*$ and $**3$ will be rainbow. Then, $L_2$ must be $s_3$, and verifying the 10 combinatorial lines that are not pillar lines and span the layers shows that this coloring would be rainbow free (since $L_3$ is forced to be $s_1$). Now let $L_1$ be $s_3$. Then, $L_2$ cannot be $s_2$, or else we would have the lines $*2*$ and $**2$ will be rainbow. Then, $L_2$ must be $s_1$, and we can once again easily verify this coloring is rainbow free. Thus, out of the 6 possible combinations of layers, we only have two rainbow free colorings. The full list of colorings is shown in the below. We will denote the coloring on the top as "Pattern 1" and the coloring on the bottom as "Pattern 2." In Figure 8, you can view the remaining 4 combinations that do not work. The rainbow combinatorial line will be highlighted in red, with each row representing a separate 10 coloring of $[3]^3$.
\begin{figure}[h]
    \centering
    \includegraphics[width=0.75\textwidth]{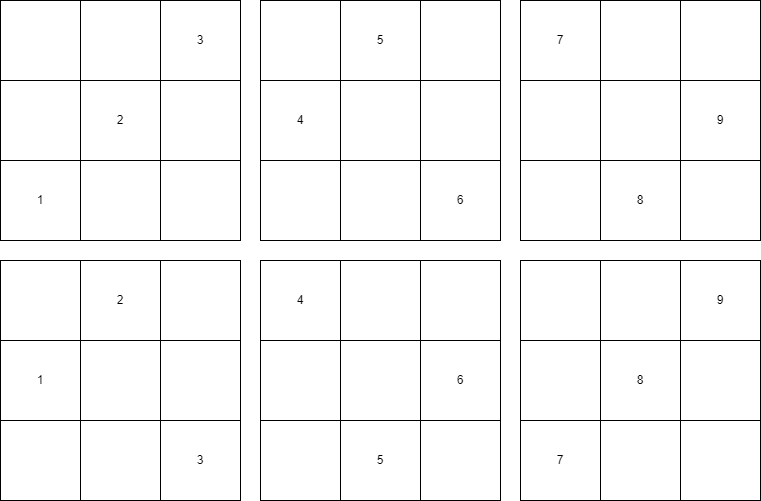}
    \caption{Two RF 10-colorings of $[3]^3$ }
    \label{fig:coloring7}
\end{figure}
\begin{figure}[h]
    \centering
    \includegraphics[width=0.75\textwidth]{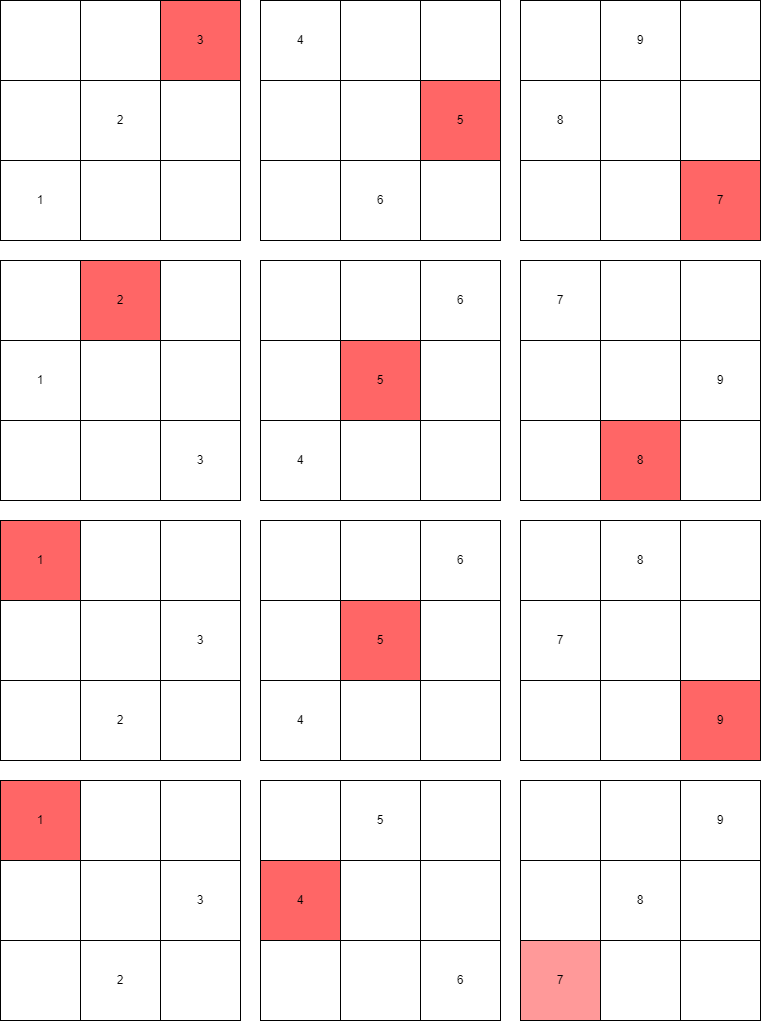}
    \caption{4 non-RF 10-colorings of $[3]^3$ }
    \label{fig:coloring15}
\end{figure}
\end{proof}
Now we know that only two RF 10-colorings of $[3]^3$ exist, we can prove properties about 27-colorings of $[3]^4$. We specifically choose 27 since it is easy to show that there are two layers that must contain 10 colors and be rainbow free, and plugging that value into Corollary 2.5 reduces the upper bound for all n $>$ 4. The main goal is to prove that a RF 27-coloring must have certain properties if it exists, then show that these properties are in contradiction with each other and the assumption that the coloring is rainbow free.
\begin{lemma}
For any RF 27-coloring of $[3]^4$, each layer can contain at most 9 distinct colors.
\end{lemma}
\begin{proof}
For the sake of contradiction, suppose we have a coloring where a layer contains 10 distinct colors. Since $ah(3,3) = 11$, we cannot have any more colors on that layer. The remaining 17 colors must be split among the last two layers, but this guarantees that a layer will have some number of distinct colors. Any pillar line containing one of those distinct colors will always be rainbow, since it contains two distinctly colored points. We arrive at a contradiction and conclude that a layer cannot contain 10 distinct colors.
\end{proof}
\begin{lemma}
For any RF 27-coloring of $[3]^4$, two layers can share at most 1 color with each other.
\end{lemma}
\begin{proof}
For contradiction, suppose we have two layers, $L_1$ and $L_2$ without loss of generality, that share two colors with each other. Let j be the number of colors shared between $L_3$ and $L_2$. The reason why we do not need to consider the number of colors shared between $L_1$ and $L_3$ is because if we can find a contradiction with $L_1$ and $L_3$ sharing no colors, we can easily find one when they do share colors (since the number of colors $L_1$ can hold is lessened). We will derive a contradiction based on j and the number of colors each layer can hold. Let $C_1, C_2, C_3$ define the maximum number of colors each layer can hold. Then, $|C_1| = 8$, $|C_2| = 8 - j$, and $|C_3| = 10 - j$. Suppose $j \geq 2$. Then, $|C_1| = 8, |C_2| \leq 6$, and $|C_3| \leq 8$. The maximum total numbers of colors all layers can hold is $C = 26 - 2j \leq 22$. Thus for all $j \geq 2$, we will have an "overflow" and have to put atleast one color onto a full layer, which guarantees 11 colors on a layer and thus a rainbow since $25 - j > C$. We now deal with the cases where $j = 0$ and $j = 1$. \\
Case 1: $j = 1$ \indent
If $j = 1$, then $|C_1| = 8, |C_2| = 7$, and $|C_3| = 9$, with $C = 24$. Since we have $25 - j = 24$ colors left, we must fill up every single layer. In other words, each layer contains 10 colors. But since each layer is also rainbow-free and there are only two RF 10-colorings of $[3]^3$, we will have two layers $L_i$ and $L_j$ with the same "pattern," where the nondominant colors are in the same position. $L_i$ and $L_j$ share at most two nondominant colors. Regardless, we can choose a pillar line in which the two points are not colored with the shared nondominant colors. Since the third layer cannot be the same pattern, or else we would automatically have some rainbow pillar line, it must be the other pattern. Then, the third point on our chosen pillar line is colored with the dominant color. If the dominant color is not a shared color, we have a rainbow. If the dominant color is shared, then either the pillar line will still be rainbow or another pillar line will be rainbow. The latter is true because if the dominant color is shared with another layer, it is either dominant or nondominant. If dominant, then the original pillar line would be rainbow, since the points on $L_i$ and $L_j$ are not colored with the dominant color. If nondominant on other layers, we can choose a line containing the dominant color and some point on $L_i$ or $L_j$ that is not colored with a shared nondominant color. We know this line exists because the dominant color can only appear once on $L_i$ or $L_j$ as a nondominant color, which means there are at most two out of nine pillar lines that do contain the dominant color. Either way, we have a rainbow pillar line.
\\
Case 2: $j = 0$ \indent
If $j = 0$, then $L_3$ contains only distinct colors. Since $|C_1| = 8, |C_2| = 8$, and $|C_3| = 9$, with $C = 25$, the only possible coloring is if we make every layer contain the maximum amount of colors. But since there exists distinctly colored points on $L_1$ and $L_2$, any pillar line containing those points will be rainbow, as the point on $L_3$ will also be distinct. Note that we only need one distinct point from $L_1$ or $L_2$, not a line containing distinctly colored points from both. \\
In all cases, we have rainbow lines and thus contradictions. Therefore, no two layers can share more than one color with each other.	
\end{proof}
\begin{lemma}
For any RF 27-coloring of $[3]^4$, all three layers must share exactly one common color. 
\end{lemma}
\begin{proof}
We know that the layers must share some colors, as if no layers share colors, any pillar line would be rainbow. Now for the sake of contradiction, suppose we have layers $L_1, L_2,$ and $L_3$ where $L_1$ shares the color $c_{12}$ with $L_2$ and $L_2$ shares another color $c_{23}$ with $L_3$. We then have $|C_1| = 9$, $|C_2| = 8$, and $|C_3| = 9$ with 25 colors left. We notice that even if $L_3$ and $L_1$ share a color, each $L_i$ must contain at least 8 distinct colors. Furthermore, two of the layers $L_i$ and $L_j$ will be RF 10-colorings, and as a result, they must be different patterns, or else we can find pillar lines with distinctly colored points on $L_i$ and $L_j$. There are two cases: $L_1$ or $L_3$ and $L_2$ are the RF 10-colorings, or $L_1$ and $L_3$ are the RF 10-colorings. \\
Case 1: $L_1$ or $L_3$ and $L_2$ are the RF 10-colorings. \indent Without loss of generality, suppose that $L_1$ is the RF 10-coloring. If $c_{12}$ is not dominant in at least one of the layers, we are essentially done. The shared color only appears once on $L_1$ or $L_2$ say at $p_1$. Then, there are 26 pillar lines that do not contain $p_1$, all of which contain two points on $L_1$ and $L_2$ that are different colors. Even supposing that $L_1$ and $L_3$ share a color, there will still be atleast 5 pillar lines to choose from (which may be a lot more if the shared color does not appear many times on $L_3$). If $c_{12}$ is dominant on both of the layers, we can once again consider the pillar lines. In this subcase, there are only 18 pillar lines with different colors on $L_1$ and $L_2$, and if these are all rainbow-free, we then have combinatorial lines that have to be rainbow on all 6 possible arrangements of patterns on and placements of $L_1$ and $L_2$, which can be found in Appendix A due to the large size and number of images. \\
Case 2: $L_1$ and $L_3$ are the RF 10-colorings. If $L_1$ and $L_3$ share a color, this becomes case 1 since $L_2$ must now also be a RF 10-coloring. We assume that $L_1$ and $L_2$ share no colors. But this is also a contradiction; any pillar line containing a distinct color on $L_2$ is rainbow, since the points on $L_1$ and $L_3$ cannot possibly be the same. \\
In both cases, we arrive at a contradiction to the assumption that this 27-coloring is RF. Thus, all three layers must share exactly one common color. 
\end{proof}
\begin{theorem}
$ah(3, 4) \leq 27$ and $ah(3, n) \leq 3^{n-1} - 2\cdot3^{n-4} + 2$ for all natural numbers n $>$ 4.
\end{theorem}
\begin{proof}
For contradiction, assume $\exists$ a RF 27-coloring of $[3]^4$. Then, all layers must share a color and $|C_1| = |C_2| = |C_3| = 9$. The remaining 26 colors will be spread among the three layers, guaranteeing two RF 10-colorings, say on $L_1$ and $L_2$. Once again, these colorings must have different patterns. If the shared color is not dominant on $L_1$ or $L_2$, we have atleast 26 pillar lines containing differently colored points from $L_1$ and $L_2$. Since $L_3$ has 8 distinct colors, we know that there must be at least 7 rainbow pillar lines. Thus, the shared color must be dominant on $L_1$ and $L_2$. We can now refer to Appendix A to show that there must always be a rainbow combinatorial line, which is a contradiction. Thus, there does not exist a RF 27-coloring of $[3]^4$. In other words, $ah(3, 4) \leq 27$ and applying Corollary 2.5 gives us $ah(3, n) \leq 3^{n-1} - 2\cdot3^{n-4} + 2$ for all natural numbers n $>$ 4.
\end{proof}
Since we have a RF 23-coloring of $[3]^4$, shown in Figure \ref{fig:coloring14}, the following corollary is true:
\begin{corollary}
$24 \leq ah(3, 4) \leq 27$.
\end{corollary}
\begin{figure}[h]
    \centering
    \includegraphics[width=0.75\textwidth]{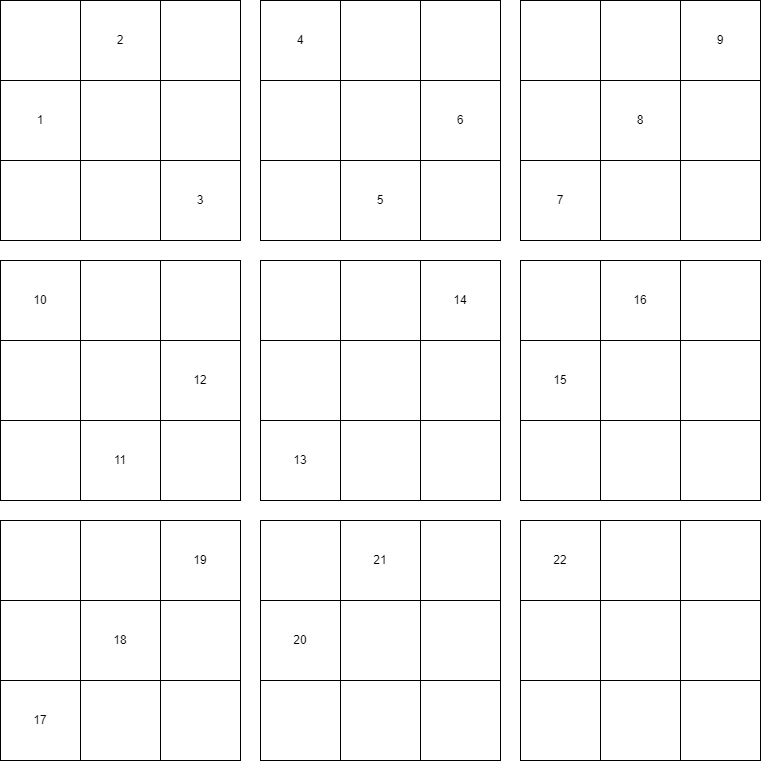}
    \caption{RF 23-coloring of $[3]^4$}
    \label{fig:coloring14}
\end{figure}
If $ah(3,4)$ is less than 27, we can obtain an even better general bound. However, using similar methods to achieve better bounds on $ah(3,4)$ may not be fruitful, as the key element to show that $ah(3,4) \leq 27$ is being able to force two layers to contain exactly 10 colors, which is extremely hard for numbers less than 27.

\section{Future Areas of Inquiry}
Since our upper bound is based on recursion applied to the anti-Hales Jewett Number for dimension 1, as we find higher dimensional colorings that are not equal to the bound, we can once again improve the upper bound using recursion. However, it may be much more enlightening to try to prove an upper bound also in terms of $(k-1)^n$ or $(k-1)^{n+1}$ to get some asymptotic bound. Furthermore, the probabilistic method could most likely be applied to help find better lower and upper bounds. \\
In a similar vein to anti-Ramsey Theory, we may also want to consider k-rainbow combinatorial lines, where a single color may appear up to k many times. This may be more helpful and interesting for larger lengths of the hypercube.\\
Questions about balanced colorings of the hypercube have already been answered by Amanda Montejano in \cite{montejano_rainbow_2024}, and they also provide some further directions related to balanced colorings and the balanced upper chromatic number of the hypercube.
\section{Acknowledgements}
I would like to thank Peter Johnson, Wenshi Zhao, who provided the construction of the nonrecursive lower bound for $k = 3$, and Joseph Briggs, who brought this problem to my attention.
\bibliographystyle{plain}
\bibliography{refs}

\begin{thebibliography}{1}

\bibitem{berikkyzy_anti-van_2017}
Zhanar Berikkyzy, Alex Shulte, and Michael Young.
\newblock Anti-van der {Waerden} {Numbers} of 3-{Term} {Arithmetic}
  {Progression}.
\newblock {\em The Electronic Journal of Combinatorics}, 24(2):P2.39, June
  2017.

\bibitem{lee_ramsey_2009}
Michelle Lee.
\newblock Ramsey {Theory}: {Van} {Der} {Waerden}’s {Theorem} and the
  {Hales}-{Jewett} {Theorem}.
\newblock August 2009.

\bibitem{montejano_rainbow_2024}
Amanda Montejano.
\newblock Rainbow considerations around the {Hales}-{Jewett} theorem, 2024.
\newblock Version Number: 1.

\bibitem{naslund_hales-jewett_2013}
Marcus Näslund.
\newblock The {Hales}-{Jewett} theorem and its application to further
  generalisations of m, n, k-games.
\newblock 2013.

\end{thebibliography}

\newpage
\appendix
\section{The 6 Arrangements of Patterns 1 and 2 in $[3]^4$}
The colorings will be arranged with each layer of $[3]^3$ as a horizonal set of three $[3]^2$. Unlike in some of the proofs, these layers will be ordered, as in $L_1$ refers to every word that starts with the letter "1." Furthermore, each column is also arranged in ascending order from left to right; the words in the leftmost column are comprised of every word that has "1" as the second letter. \\
Each color's first digit represents which layer it is on and the second digit signifies the number of distinct colors. The uncolored portions will be arbitrary, but as explained in the proof, these sections on the layers with patterns will be the dominant color in this case. \\
From these figures, you can see that there is always atleast one point/word where coloring it with any color will cause a rainbow and contradict the assumption. These points are colored purple, and the two combinatorial lines it is a part of will be colored red and blue respectively.
\begin{figure}[h]
    \centering
    \includegraphics[width=0.75\textwidth]{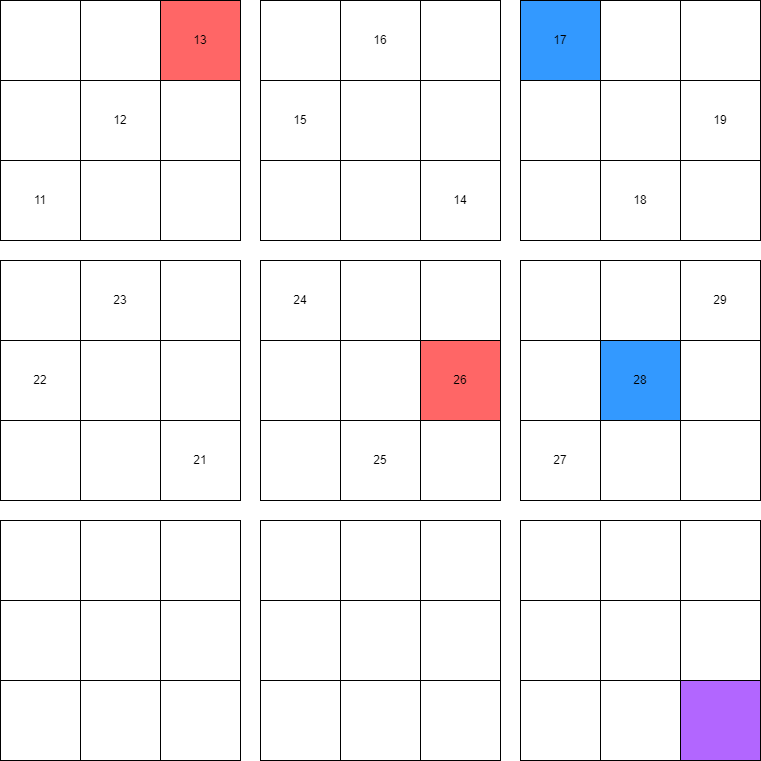}
    \caption{Pattern 1 on $L_1$ and Pattern 2 on $L_2$}
    \label{fig:coloring8}
\end{figure}
\begin{figure}[h]
    \centering
    \includegraphics[width=0.75\textwidth]{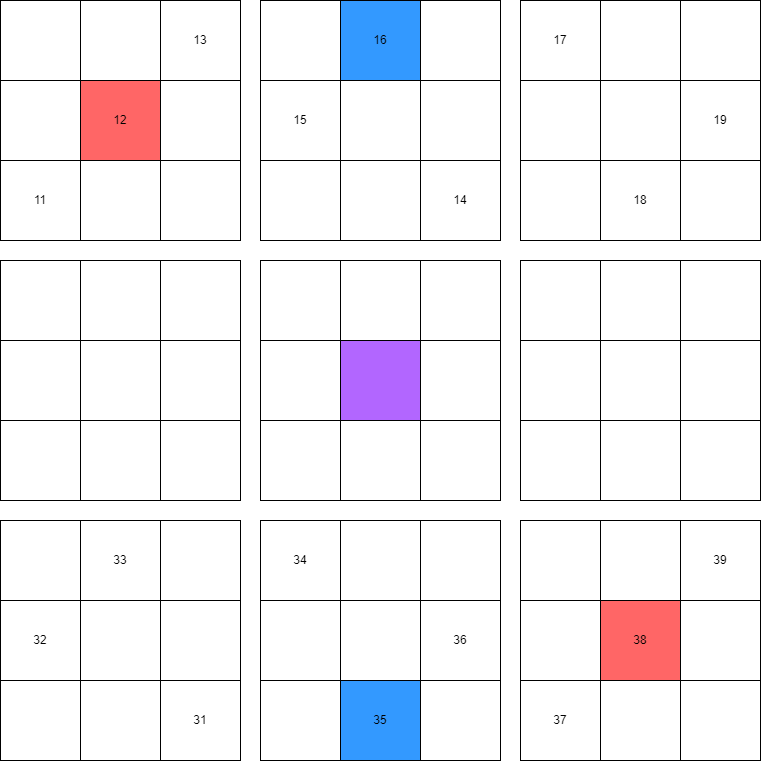}
    \caption{Pattern 1 on $L_1$ and Pattern 2 on $L_3$}
    \label{fig:coloring9}
\end{figure}

\begin{figure}[h]
    \centering
    \includegraphics[width=0.75\textwidth]{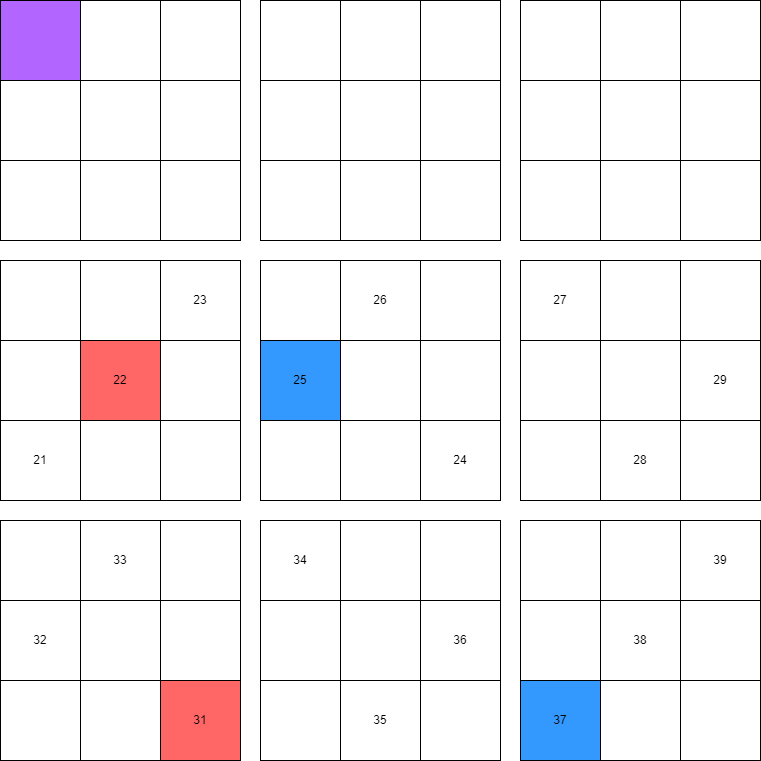}
    \caption{Pattern 1 on $L_2$ and Pattern 2 on $L_3$}
    \label{fig:coloring10}
\end{figure}
\begin{figure}[h]
    \centering
    \includegraphics[width=0.75\textwidth]{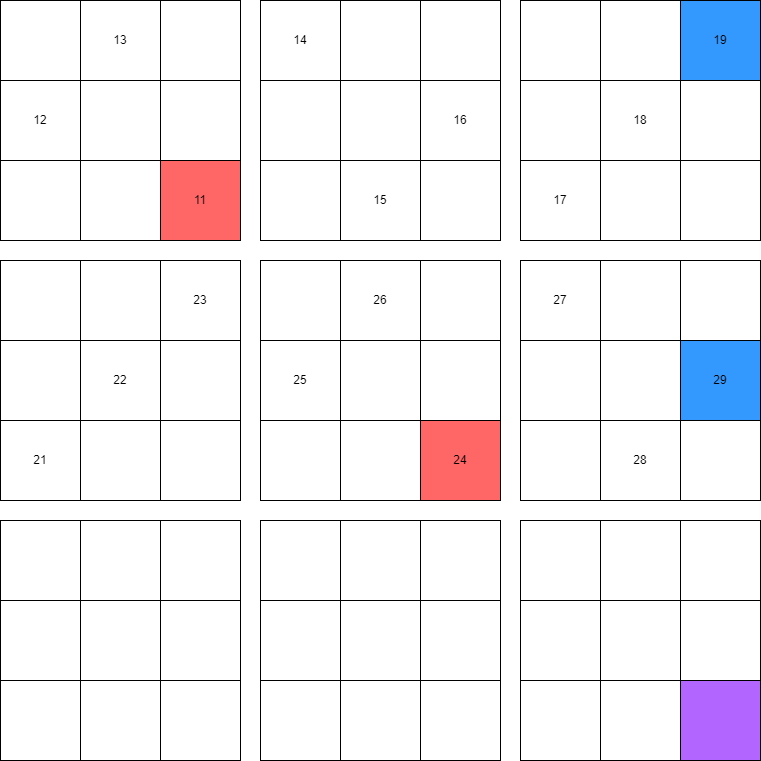}
    \caption{Pattern 2 on $L_1$ and Pattern 1 on $L_2$}
    \label{fig:coloring11}
\end{figure}
\begin{figure}[h]
    \centering
    \includegraphics[width=0.75\textwidth]{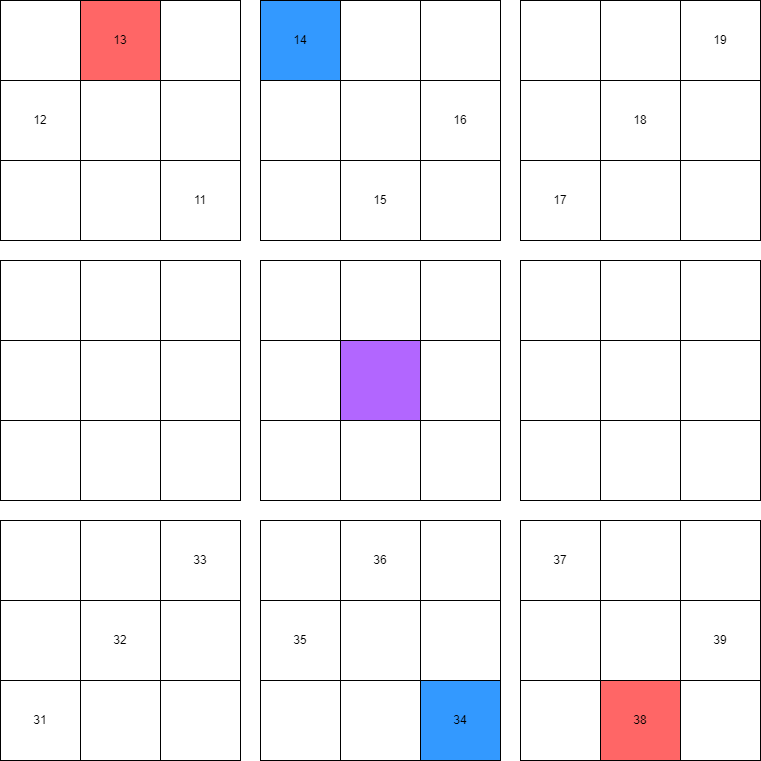}
    \caption{Pattern 2 on $L_1$ and Pattern 1 on $L_3$}
    \label{fig:coloring12}
\end{figure}
\begin{figure}[h]
    \centering
    \includegraphics[width=0.75\textwidth]{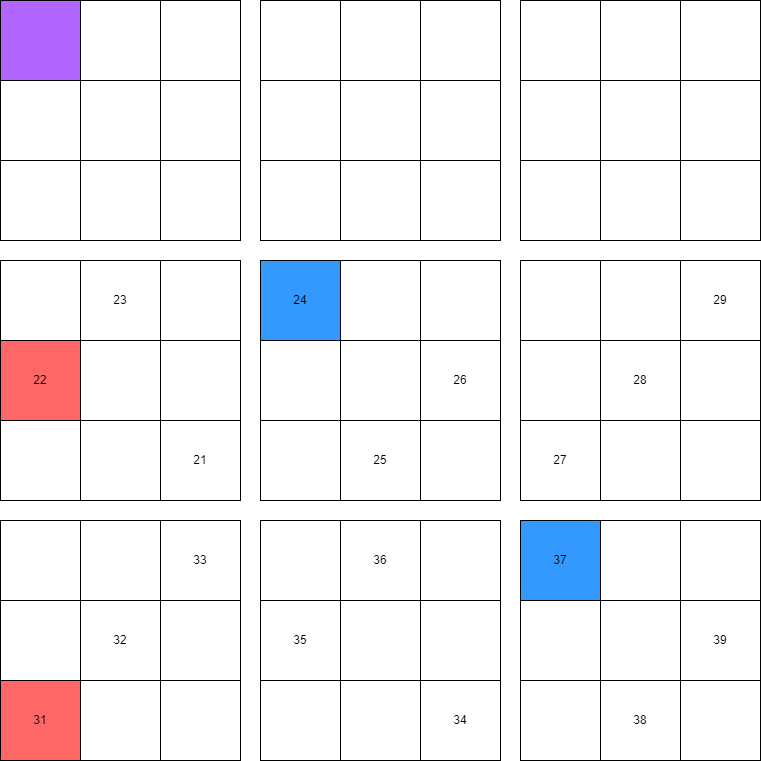}
    \caption{Pattern 2 on $L_2$ and Pattern 1 on $L_3$}
    \label{fig:coloring13}
\end{figure}

\end{document}